\documentclass[11pt,letterpaper]{amsart}
\pdfoutput=1

\usepackage{amsopn,amssymb}
\usepackage{eucal}
\usepackage{nicefrac}
\usepackage{lpic}
\usepackage[all]{xy}
\usepackage{float}
\usepackage{caption}
\usepackage{subcaption}
\usepackage{hyperref}
\hypersetup{pdfauthor={Jonah Gaster},pdftitle={Infima of length functions and dual cube complexes},colorlinks=true,linkcolor=black,citecolor=black}

\title[Infima of length functions and dual cube complexes]{Infima of length functions and\\ dual cube complexes}
\author{Jonah Gaster}
\email{gaster@bc.edu}

\address{
Department of Mathematics\\
Boston College\\ \newline
Chestnut Hill, MA 02467 \\
United States
}
\urladdr{http://www2.bc.edu/jonah-gaster}

\newcommand{\C}{\mathbb{C}}

\renewcommand{\H}{\mathbb{H}}

\newcommand{\T}{\mathcal{T}}

\newcommand{\sep}{\bold{sep}}

\newcommand{\into}{\hookrightarrow}
\renewcommand{\H}{\mathbb{H}}

\numberwithin{equation}{section}

\theoremstyle{plain}
\newtheorem{thm}{Theorem}[section]
\newtheorem*{thm*}{Theorem}

\newtheorem{lemma}[thm]{Lemma}
\newtheorem{prop}[thm]{Proposition}

\newtheorem{bigthm}{Theorem}

\theoremstyle{definition}
\newtheorem{question}{Question}

\theoremstyle{definition}
\newtheorem{definition}{Definition}

\theoremstyle{definition}
\newtheorem{example}{Example}

\theoremstyle{definition}
\newtheorem{remark}{Remark}

\begin{document}

\begin{abstract} 
In the presence of certain topological conditions, we provide lower bounds for the infimum of the length function associated to a collection of curves on Teichm\"{u}ller space that depend on the dual cube complex associated to the collection, a concept due to Sageev. As an application of our bounds, we obtain estimates for the `longest' curve with $k$ self-intersections, complementing work of Basmajian \cite{basmajian}.
\end{abstract}
\maketitle


Let $\Sigma$ be an oriented topological surface of finite type. We denote the Teichm\"{u}ller space of $\Sigma$ by $\T(\Sigma)$, which we interpret as the deformation space of marked hyperbolic structures on $\Sigma$. Given $X\in\T(\Sigma)$ and a free homotopy class (or \emph{closed curve}) $\gamma$ on $\Sigma$, we denote by $\ell(\gamma,X)$ the length of the geodesic representative of $\gamma$ in the hyperbolic structure determined by $X$. If $\Gamma= \{\gamma_i\}$ is a collection of closed curves, then we define $\ell(\Gamma,X)=\sum \ell(\gamma_i,X).$

In this note, we are concerned with translating topological information of $\Gamma$ into quantitative information about the length function $\ell(\Gamma,\cdot) : \T(\Sigma)\to\mathbb{R}$. In particular, we develop tools to estimate the infimum of $\ell(\Gamma,\cdot)$ over $\T(\Sigma)$. This work naturally complements \cite{basmajian}, where such estimates are obtained that depend on the number of self-intersections of $\Gamma$. Here we consider a finer topological invariant than the self-intersection number. 

A construction of Sageev \cite{sageev} associates to a curve system $\Gamma$ an isometric action of $\pi_1\Sigma$ on a finite-dimensional cube complex, or the \emph{dual cube complex} $\mathcal{C}(\Gamma)$ of $\Gamma$. In what follows we connect geometric properties of the dual cube complex $\mathcal{C}(\Gamma)$ to the length of the collection of curves $\Gamma$ on any hyperbolic surface. Indeed, \cite[Thm.~3]{aougab-gaster} suggests that any such information is implicitly contained in the combinatorics of $\mathcal{C}(\Gamma)$. We have:

\begin{bigthm}
\label{main estimate}
Suppose that the action of $\pi_1\Sigma$ on $\mathcal{C}(\Gamma)$ has a set of cubes $C_1,\ldots,C_m$, of dimensions $n_1, \ldots, n_m$, respectively, in distinct $\pi_1\Sigma$-orbits, and so that the union of orbits $\displaystyle \bigcup_i \pi_1\Sigma\cdot C_i$ is hyperplane separated. Then 
\[
\inf_{X\in\T(\Sigma)} \ell(\Gamma,X) \ \ge \sum_{i=1}^{m} n_i \log \left( \frac{1+\cos\frac{\pi}{n_i}}{1-\cos \frac{\pi}{n_i}}\right)~.
\]
\end{bigthm}

The definition of \emph{hyperplane separated} can be found in \S\ref{cube section}. The main use of this idea is that it allows one to conclude that `large chunks' of the preimage of the curves $\Gamma$ in the universal cover embed on the surface under the covering map. The proof of Theorem \ref{main estimate} proceeds by minimizing the length of these large chunks. 

\begin{remark}
The 
contribution to the bound above from a cube $C_i$ is useless when $C_i$ is a $2$-cube. On the other hand, it still seems reasonable to expect a lower bound on the length function in the presence of many maximal 2-cubes. Note that the presence of $m$ $2$-cubes contributes $m$ to the self-intersection number, so that Basmajian's bounds immediately imply a lower bound for the length function that is logarithmic in $m$. While Basmajian's lower bounds are sharp, the examples that demonstrate sharpness have high-dimensional dual complexes. We expect a positive answer to the following:
\end{remark}

\begin{question}
If $\mathcal{C}(\Gamma)$ contains $m$ maximal $2$-cubes, is there a lower bound for the length function $\ell(\Gamma,X)$ that is linear in $m$?
\end{question}

\begin{remark}
The bounds in Theorem \ref{main estimate} are sharp in the following sense: For each $n\in \mathbb{N}$, there exists a set of curves $\Gamma_n$ on the $(n+1)$-holed sphere $\Sigma_{0,n+1}$, and hyperbolic structures $X_n \in \T(\Sigma_{0,n+1})$ so that: 
\begin{enumerate}
\item The dual cube complex $\mathcal{C}(\Gamma_n)$ has a hyperplane separated $n$-cube. 
\item The hyperbolic length $\ell(\Gamma_n,X_n)$ is asymptotic to $n\log n$.
\end{enumerate}
\end{remark}

The problem remains of determining when a collection of curves gives rise to hyperplane separated orbits of cubes in the action of $\pi_1\Sigma$ on the dual cube complex. We offer a sufficient condition below which applies in many cases, toward which we fix some terminology.  Recall that a \emph{ribbon graph} is a graph with a cyclic order given to the oriented edges incident to each vertex. A ribbon graph $G$ is \emph{even} if the valence of each vertex is even. When an even ribbon graph $G$ is embedded on a surface $\Sigma$, a collection of homotopy classes of curves is determined by $G$ by `going straight' at each vertex. See \S\ref{ribbon section} for a more precise description.

\begin{bigthm}
\label{ribbon graph estimate}
Suppose that $G \into \Sigma$ is an embedding of an even ribbon graph $G$ into $\Sigma$ with vertices of valence $n_1, \ldots, n_m$, such that the complement $\Sigma \setminus G$ contains no monogons, bigons, or triangles. Let $\Gamma$ indicate the union of the closed curves determined by $G$. Then $G$ is a minimal position realization of $\Gamma$, the self-intersection of $\Gamma$ is given by $\binom{n_1}{2} + \ldots + \binom{n_m}{2}$, and $\mathcal{C}(\Gamma)$ contains cubes $C_1,\ldots,C_m$ of dimensions $n_1,\ldots,n_m$, respectively, in distinct $\pi_1\Sigma$-orbits, whose union is hyperplane separated.
\end{bigthm}

This provides a general method to construct curves with definite self-intersection number and definite hyperplane separated cubes in their dual cube complexes. For example:

\begin{example}
Consider the curve in Figure \ref{Gamma_tau_6_gluing}. Theorem \ref{ribbon graph estimate} implies that the curve has six hyperplane separated 3-cubes. The estimate in Theorem \ref{main estimate} now applies, so that the length of the pictured curve is at least $6\log 3$ in any hyperbolic metric on $\Sigma_6$.
\end{example}\vspace{.2cm}

Let $\mathfrak{C}_{k}(\Sigma)$ indicate curves on $\Sigma$ with self-intersection number $k$. Basmajian examined the following quantities, showing that they both are asymptotic to $\log k$:
\begin{align*}
m_{k}(\Sigma)&:=\min_{\gamma\in\mathfrak{C}_k(\Sigma)} \inf\{\ell(\gamma,X): X\in\T(\Sigma)\} \\
M_{k}&:=\;\inf \;\{m_{k}(\Sigma):\Sigma \text{ is a finite-type surface with }\chi(\Sigma)<0\}. 
\end{align*}
Note that, for each $k$ and $\Sigma$, there are finitely many mapping class group orbits among $\mathfrak{C}_k(\Sigma)$. This justifies the use of minimum in the definition of $m_{k}(\Sigma)$ above. One may define analogously:
\begin{align*}
\overline{m}_{k}(\Sigma)&:=\max_{\gamma\in\mathfrak{C}_k(\Sigma)} \inf\{\ell(\gamma,X): X\in\T(\Sigma)\} \\
\overline{M}_{k}&:=\;\sup \;\{\overline{m}_{k}(\Sigma):\Sigma \text{ is a finite-type surface with }\chi(\Sigma)<0\}. 
\end{align*}

The curves that realize the minima $m_{k}(\Sigma)$ and $M_{k}$ manage to gain a lot of self-intersection while remaining quite short, which they achieve by winding many times around a very short curve. By constructing explicit families of curves that behave quite differently---namely, they return many times to a fixed `small' compact set on the surface---we provide a lower bound for $\overline{M}_{k}$ that grows faster than Basmajian's bounds for the `shortest' curves with $k$ self-intersections.

\begin{bigthm}
\label{corollary of main thm}
We have the estimate 
\[
\limsup_{k\to\infty} \; \frac{\overline{M}_{k}}{k} \ge \frac{\log 3}{3}~.
\]
\end{bigthm}

\begin{remark}
It is not hard to observe that 
\[
\limsup_{k\to\infty} \; \frac{\overline{m}_{k}(\Sigma)}{\sqrt{k}} >0~.
\]
Indeed, given any $k$-curve $\gamma\in \mathfrak{C}_k(\Sigma)$, 
consider the closed curve $\gamma^n$ given by wrapping $n$ times around $\gamma$. The infimum of the length function of $\gamma^n$ will grow linearly in $n$, while the self-intersection number will grow quadratically in $n$. Performing this calculation with a curve with one self-intersection, one finds that 
\[
\limsup_{k\to\infty} \; \frac{\overline{m}_{k}(\Sigma)}{\sqrt{k}} \ge 4\log(1+\sqrt{2})~.
\]
\end{remark}\vspace{.2cm}

The problem of sharpness for our examples, namely good upper bounds for $\overline{m}_{k}(\Sigma)$ and $\overline{M}_{k}$, seems subtle. In particular, such upper bounds would imply an asymptotically good answer to the following question:

\begin{question}
{Given a curve $\gamma\in\mathfrak{C}_k(\Sigma)$, what is an explicit function $C(k,\Sigma)$ so that there is a point $X\in\T(\Sigma)$ with $\ell(\gamma,X)\le C(k,\Sigma)$?} 
\footnote{While this paper was under review, this question has been given an answer in \cite[Theorem 1.4]{agps}.}
\end{question} 

Note that one could also ask for an upper bound that is independent of $\Sigma$, towards which the lower bound in Theorem \ref{corollary of main thm} is more relevant.

\textbf{Organization} \ In \S\ref{cube section} we briefly recall Sageev's construction, and define hyperplane separation. In \S\ref{proof of main estimate} we lay out the necessary tools for the proof of Theorem \ref{main estimate}, and in \S\ref{proof of prop portions inject} and \S\ref{Lemma optimization} we prove these tools. \S\ref{immersed bigon section} describes a straightforward method of detecting self-intersection and hyperplane separation, and \S\ref{ribbon section} introduces even ribbon graphs and the proof of Theorem \ref{ribbon graph estimate}. Finally, \S\ref{examples section} describes a family of examples to which these tools apply, and contains a proof of Theorem \ref{corollary of main thm}.

\textbf{Acknowledgements} \ The author thanks Ara Basmajian, Martin Bridgeman, Spencer Dowdall, and David Dumas for helpful conversations, 
and Ian Biringer for a correction and reference in regard to Lemma \ref{min pos immersed bigons}.

\section{Dual cube complexes and hyperplane separation}
\label{cube section}
We recall Sageev's construction. 
A collection $\Gamma$ of homotopy classes of curves on $\Sigma$ gives rise to an isometric action of $\pi_{1}\Sigma$ on a CAT(0)-cube complex $\mathcal{C}(\Gamma)$. 
This action is obtained roughly as follows: 
Choose a minimal position realization $\lambda$ of the curves in $\Gamma$, and 
consider the preimage $\widetilde{\lambda}$ of $\lambda$ in the universal cover $\widetilde{\Sigma}$. 
In the language of \cite{Wise}, the set $\widetilde{\lambda}$ decomposes into a union\footnote{An illustrative example is provided by the case that $\lambda$ is given by the geodesic representatives of $\Gamma$ relative to a chosen hyperbolic structure, in which case $\widetilde{\lambda}\subset \mathbb{H}^2$ is evidently a union of complete geodesics.} of \emph{elevations}.
Each elevation splits $\widetilde{\Sigma}$ into two connected components. 
A \emph{labeling} of $\tilde{\lambda}$ is a choice of half-space in the complement of each of the elevations. 
The one skeleton of the cube complex $\mathcal{C}(\Gamma)$ is built from \emph{admissible} labelings of $\widetilde{\lambda}$, i.e.~choices of half-spaces in the complement of the elevations so that any pair intersect. 

Two such admissible labelings are connected by an edge when they differ on precisely one elevation of $\widetilde{\lambda}$ (in analogy with the dual graph to the set $\widetilde{\lambda}$).
Finally, $\mathcal{C}(\Gamma)$ is given by the unique  non-positively curved cube complex with the prescribed 1-skeleton.
The action of $\pi_1\Sigma$ on the elevations comprising $\widetilde{\lambda}$ naturally induces a permutation of the labelings, which induces an isometry of $\mathcal{C}(\Gamma)$.
See \cite{sageev}, \cite{sageev2}, \cite{chatterji-niblo} for details.

We collect this information conveniently:

\begin{thm*}[Sageev]
\label{concise sageev}
The action of $\pi_1\Sigma$ on the $\mathrm{CAT}(0)$ cube complex $\mathcal{C}(\Gamma)$ is independent of realization. There is a $\pi_{1}\Sigma$-equivariant incidence-preserving correspondence of the hyperplanes of $\mathcal{C}(\Gamma)$ with the elevations in $\widetilde{\lambda}$, so that maximal $n$-cubes are in correspondence with maximal collections of $n$ pairwise intersecting elevations of curves in $\Gamma$.
\end{thm*}

\noindent In light of Sageev's theorem we may sometimes identify the elevations in $\widetilde{\lambda}$ with the hyperplanes of the cube complex $\mathcal{C}(\Gamma)$.

Given a cube $C$ in a cube complex $\mathcal{C}$, we denote the set of hyperplanes of $\mathcal{C}$ by $\mathcal{H}(\mathcal{C})$, and the set of hyperplanes of $C$ by $\mathcal{H}(C)\subset\mathcal{H}(\mathcal{C})$.

\begin{definition}
\label{separated}
{Suppose $C$ and $D$ are two cubes in a cube complex. We say that $C$ and $D$ are \emph{hyperplane separated} if either $\mathcal{H}(C) \cap \mathcal{H}(D) = \emptyset$, or $|\mathcal{H}(C) \cap \mathcal{H}(D)| = 1$, and, for any $c\in\mathcal{H}(C)$ and $d\in\mathcal{H}(D)$ with $c\ne d$, the hyperplanes $c$ and $d$ are disjoint.}
\end{definition}

A union of cubes $\bigcup_i C_i$ is hyperplane separated when every pair is hyperplane separated.

\section{Proof of Theorem \ref{main estimate}}
\label{proof of main estimate}

Consider an $n$-cube $C\subset\mathcal{C}(\Gamma)$. The orientation of $\widetilde{\Sigma}$ induces a counter-clockwise cyclic ordering of the $n$ elevations of curves from $\Gamma$ that correspond to the hyperplanes of $C$. In what follows, we fix a hyperbolic surface $X\in\T(\Sigma)$ and identify the universal cover $\widetilde{\Sigma}$ with $\H^{2}$. We will work with the Poincar\'{e} disk model for $\H^2$, with conformal boundary $S^1$. 

Enumerate the $n$ geodesic representatives $(\gamma_{1},\ldots,\gamma_{n})$ of the elevations of curves that correspond to $C$, respecting the cyclic order. Each geodesic $\gamma_i$ has two endpoints $p_i,q_i\in S^1$. Choose these labels so that $p_1,\ldots,p_n,q_1,\ldots,q_n$ is consistent with the cyclic order of $S^1$.

Given a trio of elevations $\gamma_{i-1},\gamma_{i},\gamma_{i+1}\subset\H^{2}$, consider the pair of distinct disjoint geodesics $(p_{i-1},p_{i+1})$ and $(q_{i-1},q_{i+1})$. 
We will refer to this pair as the \emph{separators} of the hyperplane $\gamma_{i}$ in $C$, and we will denote the pair by $\sep(\gamma_{i},C)$. 
(When $i=1$ or $i=n$, the separators of $\gamma_i$ are the geodesics $(p_{i-1},q_{i+1})$ and $(q_{i-1},p_{i+1})$, with indices read modulo $n$.)
Let $\delta_{i}$ indicate the portion of $\gamma_{i}$ between the separators. We will refer to the arcs $\{\delta_{1},\ldots,\delta_{n}\}$ as the \emph{diagonals} of the cube $C$. See Figure \ref{n-cubePic2} for a schematic picture. 

\begin{figure}[h]
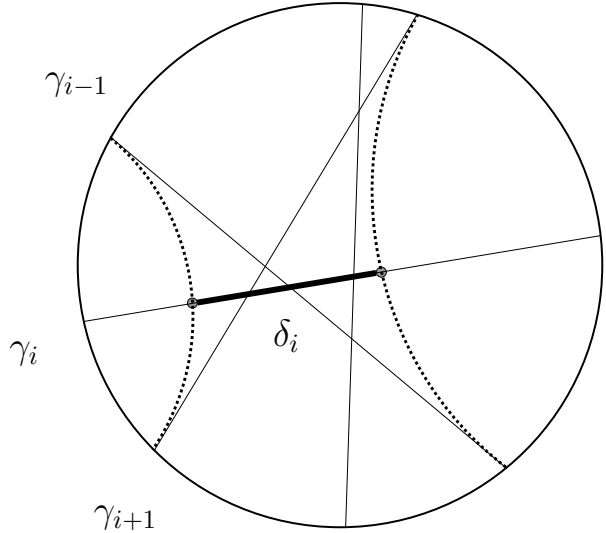

	\centering
	\begin{lpic}[clean]{n-cube2(,7cm)}
		\Large
		\lbl[]{-15,50;$\gamma_{i}$}
		\lbl[]{14,3;$\gamma_{i+1}$}
		\lbl[]{0,127;$\gamma_{i-1}$}
		\lbl[]{60,55;$\delta_{i}$}
	\end{lpic}
	\caption{The separators of the hyperplane $\gamma_{i}$ and the diagonal $\delta_{i}$.}
	\label{n-cubePic2}
\end{figure}

\begin{lemma}
\label{delta lengths}
For each $i$, we have 
\[
\ell(\delta_{i}) = \log \left| \frac{(p_i - q_{i-1} ) ( p_i - q_{i+1} )(q_i - p_{i+1})(q_i - p_{i-1})}{(p_i - p_{i+1})(p_i - p_{i-1})(q_{i} - q_{i+1})(q_i - q_{i-1})} \right|~.
\]
\end{lemma}

\begin{proof}
A calculation in $\H^2$.
\end{proof}

Towards Theorem \ref{main estimate}, we suppose below that $C_1,\ldots,C_m$ are cubes of $\mathcal{C}(\Gamma)$ in distinct $\pi_1\Sigma$-orbits. Let $\delta_1^i,\ldots,\delta_{n_i}^i$ be the diagonals of the cube $C_i$, let 
\[
\mathcal{D}_i:=\displaystyle\bigcup_{k}\delta_k^i
\]
 indicate the union of the diagonals of $C_i$, and $\mathcal{D}:=\mathcal{D}_1\cup \ldots \cup \mathcal{D}_m$.

For ease in exposition, we postpone the proof of the following proposition:

\begin{prop}
\label{portions inject}
If the union of orbits $\bigcup \pi_1 \Sigma\cdot C_i$ is hyperplane separated, then the covering map $\pi:\H^2\to\Sigma$ is injective on the union $\mathcal{D}$ minus a finite set of points.
\end{prop}

Finally, we will need the solution to the following optimization problem, whose proof we also postpone: Given $2n$ distinct points $x_1,\ldots,x_{2n}\in S^1$, for notational convenience we adopt the natural convention that subscripts should be read modulo $2n$, so that $x_{2n+1}=x_1$ and $x_0=x_{2n}$. Let $F(x_1, \ldots , x_{2n})$ indicate

\[
F(x_1, \ldots, x_{2n} ) = \log \prod_{j=1}^{2n} \left| \frac{(x_j - x_{j+n+1})(x_j - x_{j+n-1})}{(x_j - x_{j+1})(x_j - x_{j-1})} \right|~.
\]

\begin{lemma}
\label{harder optimization}
When $(x_1,\ldots,x_{2n})$ are cyclically ordered in $S^1$, we have 
\[
F(x_1,\ldots,x_{2n})\ge n \log \left( \frac{1+\cos\frac{\pi}{n}}{1-\cos\frac{\pi}{n}} \right)~.
\]
\end{lemma}

Assuming for now Propostion \ref{portions inject} and Lemma \ref{harder optimization}, we are ready to prove Theorem \ref{main estimate}:

\begin{proof}[Proof of Theorem \ref{main estimate}]
{We bound from below the sum of lengths of the curves from $\Gamma$ in the hyperbolic structure determined by $X\in\T(\Sigma)$. Pull $\Gamma$ tight to geodesics, and consider the pre-image under the covering transformation. As described above, each cube $C_i$, of dimension $n_i$, has $n_i$ hyperplanes with $n_i$ corresponding elevations of mutually intersecting geodesics in $\H^2$. These curves determine $2n_i$ cyclically ordered distinct points $$p^i_1,p^i_2,\ldots,p^i_{n_i},q^i_1,q^i_2,\ldots,q^i_{n_i}$$ on $S^1$, the diagonals $\mathcal{D}_i$, and $\mathcal{D}$, the union of $\mathcal{D}_i$. We estimate: 

\begin{align*}
\ell(\Gamma,X) & \ge \ell(\mathcal{D}) \ = \ \sum_{i=1}^m \ell(\mathcal{D}_i) &\hspace{-3cm}\text{by Proposition \ref{portions inject}}\\
&= \sum_{i=1}^m \sum_{j=1}^{n_i} \ell(\delta_j^i) \\
& = \sum_{i=1}^m \log \prod_{j=1}^{n_i} \left| \frac{\left(p^i_j - q^i_{j-1} \right) \left( p^i_j - q^i_{j+1} \right)\left(q^i_j - p^i_{j+1}\right)\left(q^i_j - p^i_{j-1}\right)}{\left(p^i_j - p^i_{j+1}\right)\left(p^i_j - p^i_{j-1}\right)\left(q^i_j - q^i_{j+1}\right)\left(q^i_j - q^i_{j-1}\right)} \right| \\
&&\hspace{-3cm} \text{by Lemma \ref{delta lengths}}\\
& \ge \sum_{i=1}^{m} n_i \log \left( \frac{1+\cos\frac{\pi}{n_i}}{1-\cos\frac{\pi}{n_i}} \right), &\hspace{-3cm}\text{by Lemma \ref{harder optimization}.}
\end{align*}
}
\end{proof}

\section{Proof of Proposition \ref{portions inject}}
\label{proof of prop portions inject}

Proposition \ref{portions inject} is the sole motivation for the definition of `hyperplane separated'. We turn to the proof. To aid our exposition, we will say that an $H$ in $\H^{2}$ is a pair of disjoint geodesics, and a geodesic arc connecting them. Associated to an $H$ is a \emph{cross}, a pair of intersecting geodesics with the same limit points as the $H$. See Figures \ref{anH} and \ref{itsCross} for a schematic. For example, the union of the diagonal $\delta_{i}$ and the separators $\sep(\gamma_{i},C)$ form an $H$, with associated cross $\{\gamma_{i-1},\gamma_{i+1}\}$.

\begin{figure}[h]
\begin{minipage}[]{.28\linewidth}
	\centering
	\begin{lpic}[clean]{anH(3cm)}
		\lbl[]{180,250;$\mathrm{bar}$}
	\end{lpic}
	\subcaption{An $H$ and its bar.}
	\label{anH}
\end{minipage}
\hspace{.5cm}
\begin{minipage}[]{.28\linewidth}
	\centering
	\begin{lpic}[clean]{itsCross(3cm)}
	\end{lpic}
	\subcaption{The cross of an $H$.}
	\label{itsCross}
\end{minipage}
\hspace{.5cm}
\begin{minipage}[]{.28\linewidth}
	\centering
	\includegraphics[width=3cm]{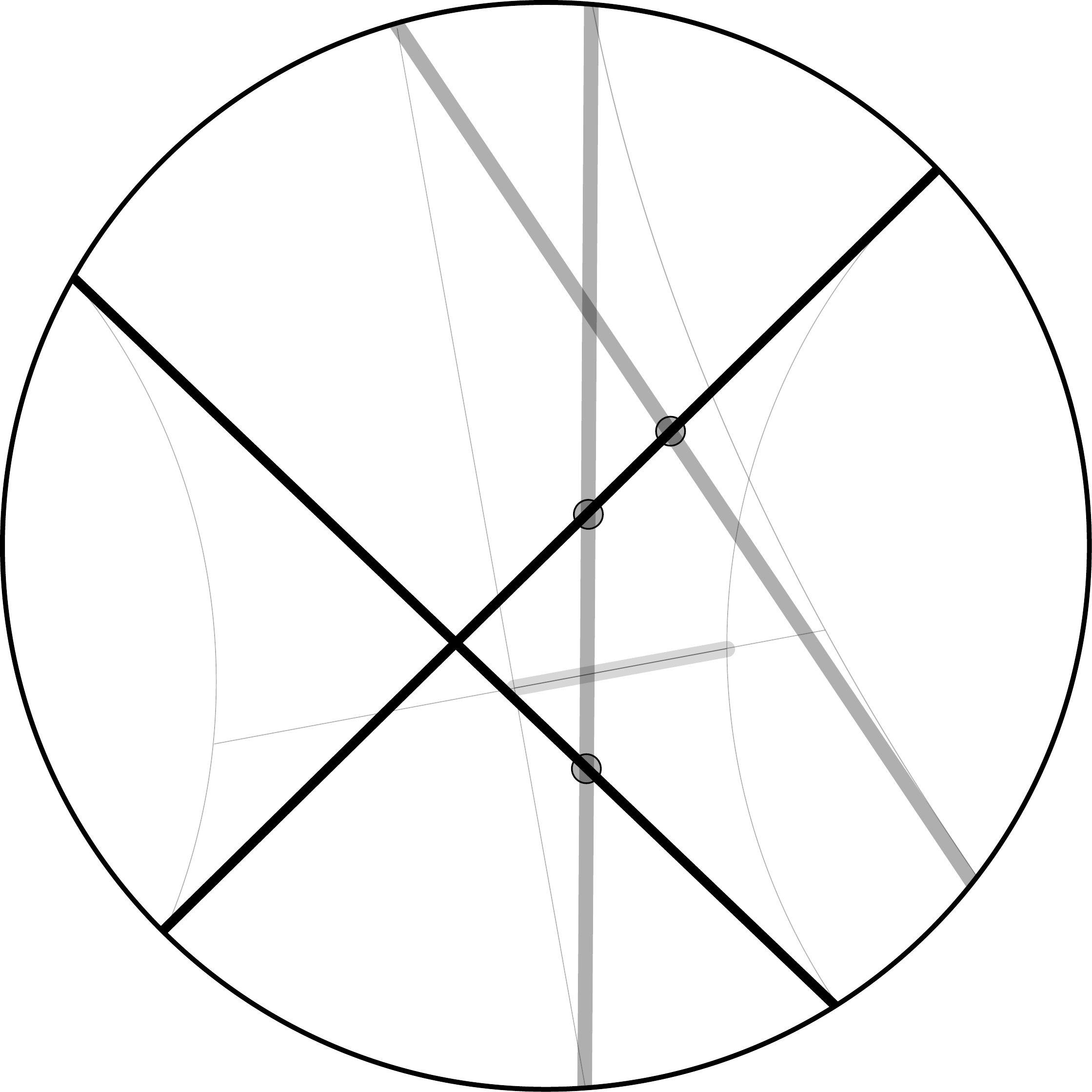}
	\subcaption{H's with overlapping bars.}
	\label{overlappingHs}
\end{minipage}
\caption{The notion of an `$H$' in $\H^2$.}
\label{Hpics}
\end{figure}

\begin{lemma}
\label{overlapping bars}
{Suppose $H_{1},H_{2}\subset\H^{2}$ are distinct $H$'s, such that their bars overlap in an interval. Then the crosses of $H_{1}$ and $H_{2}$ intersect.}
\end{lemma} 

\noindent See Figure \ref{overlappingHs} for a schematic.

\begin{proof}
{The convex hull of an $H$ is an ideal quadrilateral. By assumption, the convex hulls of $H_{1}$ and $H_{2}$ intersect. The lemma follows from the following simple observation: If two ideal quadrilaterals intersect, then their crosses intersect. We demonstrate this below. Note that the ideal points of an ideal quadrilateral are cyclically ordered. We say that two such points are \emph{opposite} if they are not neighbors in the cyclic order.

Let $P$ and $Q$ be intersecting ideal quadrilaterals, with cyclically ordered ideal points $\partial P$ and $\partial Q$ in $\partial_{\infty}\H^{2}$. As $P$ and $Q$ intersect, there are two points $q,q'\in\partial Q$ lying in distinct components of $\partial_{\infty}\H^{2}\setminus \partial P$. Suppose that $q$ and $q'$ are not opposite vertices. The vertex that follows $q'$ in the cyclic order is either in the same component of $\partial_{\infty}\H^{2}\setminus \partial P$ as $q$, in which case there are three vertices in the same component as $q$, or it is in a distinct component from $q$. Thus if $P$ and $Q$ intersect, there are a pair of opposite vertices of $\partial Q$ in distinct components of $\partial_{\infty}\H^{2}\setminus \partial P$. 

Opposite vertices of an ideal quadrilateral are boundary points of the cross of the quadrilateral. Thus there is a geodesic of the cross of $Q$ that runs between distinct components of $\partial_{\infty}\H^{2}\setminus \partial P$, so that the crosses of $P$ and $Q$ intersect.}
\end{proof}

\begin{proof}[Proof of Proposition \ref{portions inject}]
{As the union $\mathcal{D}$ is compact, properness of the action of $\pi_{1}\Sigma$ ensures that there are finitely many elements $g\in\pi_{1}\Sigma$ so that $g\cdot \mathcal{D} \cap \mathcal{D} \ne \emptyset$. Let $\mathcal{D}'$ indicate the complement in $\mathcal{D}$ of the finitely many points that are transversal intersections of $\mathcal{D}$ with $g\cdot \mathcal{D}$. If $\pi$ is not injective on $\mathcal{D}'$, then there is an element $1\ne g\in\pi_{1}\Sigma$ so that $g\cdot \delta_i^k$ and $\delta_j^l$ overlap in an open interval. In particular, note that $g$ sends the hyperplane containing $\delta_i^k$ to the hyperplane containing $\delta_j^l$. 

If $k=l$ and $g\cdot C_k = C_k$, then $g\cdot \mathcal{D}_k=\mathcal{D}_k$, and by Brouwer's fixed point theorem there would be a fixed point of $g$, violating freeness of the action of $\pi_{1}\Sigma$. Since $g\cdot C_k\ne C_l$ for $k\ne l$ (recall that the cubes $\{C_i\}$ are in distinct $\pi_1\Sigma$-orbits), we may thus assume that $g\cdot C_k$ and $C_l$ are distinct cubes that share the common hyperplane containing the diagonals $g\cdot \delta_i^k$ and $\delta_j^l$. Let the hyperplanes of $C_k$ be given by $\{\gamma_1,\ldots,\gamma_{n_k}\}$, and those of $C_l$ by $\{\eta_1,\ldots,\eta_{n_l}\}$, so that $g\cdot \gamma_i = \eta_j$. 

Observe that a trivial consequence of separatedness is that the separators $\sep(\eta_j,g\cdot C_k)$ are not the same pair of geodesics as the separators $\sep(\eta_j,C_l)$: If they were identical, then $g\cdot C_k$ and $C_l$ would be two distinct cubes in the union $\bigcup \pi_1 \Sigma\cdot C_i$ that share the hyperplanes corresponding to $\gamma_{j-1}$, $\gamma_{j}$, and $\gamma_{j+1}$.

Consider then the two $H$'s formed by $g\cdot \delta_i^k$ and $\sep(\eta_j,g\cdot C_k)$ on the one hand, and $\delta_j^l$ and $\sep(\eta_j,C_l)$ on the other. By assumption these two $H$'s have overlapping bars, so that by Lemma \ref{overlapping bars} their crosses intersect. Namely, one of $g\cdot\gamma_{i-1}$ and $g\cdot\gamma_{i+1}$ intersects one of $\eta_{j-1}$ and $\eta_{j+1}$. This contradicts separatedness of $C$. We conclude that $\pi$ is injective on $\mathcal{D}'$, the union of the diagonals of $C_1,\ldots,C_m$ minus finitely many points, as desired.}
\end{proof}

\section{Proof of Lemma \ref{harder optimization}}
\label{Lemma optimization}
We solve the necessary optimization problem:

\begin{proof}[Proof of Lemma \ref{harder optimization}]
{Note that $F$ has several useful invariance properties: First it is clear that $F$ is invariant under rotations of $S^1$. More generally, the conformal automorphisms of the disk $Aut(\mathbb{D})$ act diagonally on $(S^1)^{2n}$, and for any $\sigma\in Aut(\mathbb{D})$, $F \circ\sigma=F$. As well, it is immediate from the definition that $F(x_1,x_2,\ldots,x_{2n})=F(x_2,\ldots,x_{2n},x_1)$. 

For each $j=1,\ldots,2n$, let $x_j=e^{i\theta_j}$. Applying a rotation of $S^1$ if necessary, we assume that $0\le\theta_1 < \ldots < \theta_{2n}<2\pi$.

The identity $|e^{i\alpha} - e^{i\beta}|=\sqrt{2-2\cos(\alpha - \beta)}$ implies that 
\[
\log \left| \frac{e^{i \theta_j} - e^{i \theta_k}}{e^{i \theta_j} - e^{i \theta_l}} \right| =
\frac{1}{2} \log \frac{1-\cos(\theta_j - \theta_k)}{1-\cos(\theta_j - \theta_l)}~.
\]
Taking a derivative we find
\begin{align*}
\frac{\partial F}{\partial \theta_j} = \frac{\sin(\theta_j - \theta_{j+n-1})}{1-\cos(\theta_j - \theta_{j+n-1})} \ + \ & \frac{\sin(\theta_j - \theta_{j+n+1})}{1-\cos(\theta_j - \theta_{j+n+1})} \\ 
- \ \frac{\sin(\theta_j - \theta_{j-1})}{1-\cos(\theta_j - \theta_{j-1})} \ - \ &\frac{\sin(\theta_j - \theta_{j+1})}{1-\cos(\theta_j - \theta_{j+1})}~.
\end{align*}
Since $\displaystyle\frac{\sin \theta}{1-\cos \theta} = \cot \frac{\theta}{2}~,$
we may write the above as 
\begin{align*}
\frac{\partial F}{\partial \theta_j} = \cot \frac{\theta_j - \theta_{j+n-1}}{2} + \cot \frac{\theta_j - \theta_{j+n+1}}{2}
- \cot \frac{\theta_j - \theta_{j+1}}{2} - \cot \frac{\theta_j - \theta_{j-1}}{2}.
\end{align*}

Towards candidates for absolute minima of $F$, we seek solutions to the system of equations $\left\{\frac{\partial F}{\partial \theta_j}=0 \right\}$. Given the invariance properties of $F$, any such solution is far from unique, even locally. In order to characterize the unique $Aut(\mathbb{D})$-orbit of a solution, we pick a $j$, and fix the choices $\theta_{n+j}-\theta_j=\pi$, and $\theta_{n+j+1}-\theta_{j+1}=\pi$.

With $\theta_j+\pi$ substituted for $\theta_{n+j}$, the equations $\frac{\partial F}{\partial \theta_j}=0$ and $\frac{\partial F}{\partial \theta_{n+j}}=0$ now yield

\begin{align*}
 \cot \frac{\theta_j - \theta_{j+n-1}}{2} + \cot \frac{\theta_j - \theta_{j+n+1}}{2} & = \cot \frac{\theta_j - \theta_{j+1}}{2} + \cot \frac{\theta_j - \theta_{j-1}}{2}~, \text{ and }\\
 \tan \frac{\theta_j - \theta_{j-1}}{2} + \tan \frac{\theta_j - \theta_{j+1}}{2} & = \tan \frac{\theta_j - \theta_{j+n-1}}{2} + \tan \frac{\theta_j - \theta_{j+n+1}}{2}~.
\end{align*}
Eliminating $\displaystyle\tan{\frac{\theta_j - \theta_{j+1}}{2}}~,$ we find
\begin{align*}
\cot \frac{\theta_j - \theta_{j+n-1}}{2}& + \cot \frac{\theta_j - \theta_{j+n+1}}{2} - \cot \frac{\theta_j - \theta_{j-1}}{2} = \\
& \frac{1}{\tan \frac{\theta_j - \theta_{j+n-1}}{2} + \tan \frac{\theta_j - \theta_{j+n+1}}{2} - \tan \frac{\theta_j - \theta_{j-1}}{2}}~. 
\end{align*}

Recall the remarkable fact that the solutions of the equation 
\[
\frac{1}{x}+\frac{1}{y}+\frac{1}{z}=\frac{1}{x+y+z}
\] 
are precisely the equations $x=-y$, $x=-z$, or $y=-z$. As a consequence, we have one of the following:
\begin{align*}
& \tan \frac{\theta_j - \theta_{j+n-1}}{2} = -\tan \frac{\theta_j - \theta_{j+n+1}}{2}~,\\
& \tan \frac{\theta_j - \theta_{j+n-1}}{2} = \tan \frac{\theta_j - \theta_{j-1}}{2}~, \text{ or} \\
& \tan \frac{\theta_j - \theta_{j+n+1}}{2} = \tan \frac{\theta_j - \theta_{j-1}}{2}~.
\end{align*}
By assumption, 
\[
0< \theta_{j+n-1}-\theta_j < \pi < \theta_{j+n+1}-\theta_j < \theta_{j-1}-\theta_j < 2\pi~,
\] 
so that 
\[
-\pi < \frac{\theta_j - \theta_{j-1}}{2} < \frac{\theta_j - \theta_{j+n+1}}{2} < -\frac{\pi}{2} < \frac{\theta_j - \theta_{j+n-1}}{2} < 0~.
\]

The only possibility above is thus the first equation, so that 
\[
\frac{\theta_{n+j+1} - \theta_j}{2} - \pi = \frac{\theta_j - \theta_{j+n-1}}{2}~,
\]
or $\theta_{j+n+1}+\theta_{j+n-1} = 2\theta_j + 2\pi$. The equation $\frac{\partial F}{\partial \theta_j}=0$ now implies as well that $\theta_{j-1} + \theta_{j+1} = 2\theta_j + 2\pi$. 

On the other hand, we've also assumed that $\theta_{j+n+1}-\theta_{j+1}=\pi$, so that we have 
\[
\theta_{j+n-1}+\theta_{j+1} =2\theta_j + 2\pi - \theta_{j+n+1} + \theta_{j+1} = 2\theta_j + \pi~.
\]
This implies that 
\begin{align*}
\theta_{j-1}-\theta_{j+n-1} & = \theta_{j-1} - (2\theta_j+\pi-\theta_{j+1})\\
& = \theta_{j-1} - \theta_2 - (2\theta_j + \pi)\\
& = (2\theta_j + 2\pi) - (2\theta_j + \pi) = \pi~.
\end{align*}
Note that we now know that, if we make the normalizing assumptions that $\theta_{j+n}=\theta_j+\pi$ and $\theta_{j+n+1}=\theta_{j+1}+\pi$, then the equations $\left\{ \frac{\partial F}{\partial \theta_j}=0 , \frac{\partial F}{\partial \theta_{j+n}}=0 \right\}$ ensure that $\theta_{j-1}=\theta_{j+n-1}+\pi$. Using all of the equations $\left\{\frac{\partial F}{\partial \theta_j} = 0 \right\}$, it is now evident that $\theta_{n+k}=\theta_k + \pi$, for each $k=1,\ldots,n$.

We apply this understanding to the equation $\frac{\partial F}{\partial \theta_j} = 0$: 
\begin{align*}
\cot \frac{\theta_j - \theta_{j-1}}{2} + \cot \frac{\theta_j - \theta_{j+1}}{2} & = \cot \frac{\theta_j - \theta_{j+n-1}}{2} + \cot \frac{\theta_j - \theta_{j+n+1}}{2} ~,\\
& = \cot \frac{\theta_j - \theta_{j-1} - \pi}{2} + \cot \frac{\theta_j - \theta_{j+1} - \pi}{2}~,\\
& = - \tan \frac{\theta_j - \theta_{j-1}}{2} - \tan \frac{\theta_j - \theta_{j+1}}{2}~,
\end{align*}
so that
\[
-\tan \frac{\theta_j - \theta_{j-1}}{2} - \cot \frac{\theta_j - \theta_{j-1}}{2} = \tan \frac{\theta_j - \theta_{j+1}}{2} + \cot \frac{\theta_j - \theta_{j+1}}{2}~.
\]
Since $\displaystyle \tan x+\cot x = \frac{2}{\sin 2x}$, we obtain
\[
\sin \left( \theta_j - \theta_{j-1} \right) = \sin \left( \theta_{j+1} - \theta_j \right)~.
\]
If $\theta_j - \theta_{j-1} = \pi - (\theta_{j+1} - \theta_j)$, then $\theta_{j+1} - \theta_{j-1} = \pi$. On the other hand, $\theta_{j+n}=\theta_j + \pi$, and the $\theta_j$ are distinct. Thus $\theta_j - \theta_{j-1} = \theta_{j+1} - \theta_j$, for each $j=1,\ldots,2n.$ Setting $\theta_1=0$, we see that $\left(1,e^{\frac{\pi i}{n}}, e^{\frac{2\pi i}{n}}, \ldots, e^{\frac{(2n-1)\pi i}{n}} \right)$ is the unique $Aut(\mathbb{D})$-orbit for which the partial derivatives simultaneously vanish. As it is evident that $F(x_1,\ldots,x_{2n})$ goes to $+\infty$ as points $x_j$ and $x_{j+1}$ collide, the absolute minimum of $F$ must occur at a simultaneous zero of its partial derivatives. Evaluating $F \hspace{-.05cm} \left(1,e^{\frac{\pi i}{n}}, e^{\frac{2\pi i}{n}}, \ldots, e^{\frac{(2n-1)\pi i}{n}} \right)$ achieves the result.}\end{proof}

\section{Bigons and triangles}
\label{immersed bigon section}

Towards Theorem \ref{ribbon graph estimate}, for the computation of the self-intersection number of a self-intersecting closed curve, we require a slight generalization of the `bigon criterion' of \cite{farb-margalit}. Recall that a representative $\lambda$ of a collection of closed curves $\Gamma$ is in \emph{minimal position} if its intersection points are transverse, and the number of intersections of $\lambda$, counted with multiplicity, is minimal among representatives of $\Gamma$. A \emph{monogon} is a polygon with one side and a \emph{bigon} is a polygon with two sides.

\begin{definition}
A representative $\lambda$ of a collection of closed curves $\Gamma\subset \Sigma$ has an \emph{immersed monogon} if there is an immersion of a monogon whose boundary arc is contained in $\lambda$, and it has an \emph{immersed bigon} if there is such an immersion of a bigon.
\end{definition}

\begin{lemma}
\label{min pos immersed bigons}
If the representative $\lambda$ of a collection of closed curves on $\Sigma$ is without immersed monogons and without immersed bigons, then it is in minimal position.
\end{lemma}

\begin{remark}
An error in a previous version of this lemma was pointed out by Ian Biringer, as well as a reference to a very similar statement due to Hass and Scott. The corrected statement is above (cf.~\cite[Thms.~3.5,4.2]{hass-scott}). 
As they note, a non-primitive curve demonstrates that the converse is false \cite[p.~94]{hass-scott}.
\end{remark}

\begin{proof}
Suppose $\lambda$ is without immersed bigons or monogons, and has $n$ transverse self-intersections. Let $G_\lambda\subset \Sigma$ indicate the graph determined by $\lambda$, choose a spanning tree $T_\lambda$ for $G_\lambda$, a lift of $T_\lambda$ to the universal cover, and a representative of each of the $n$ intersection points. At each of these representative intersection points, the pre-image of $\lambda$ in $\widetilde{\Sigma}$ consists of a pair of linked curves: If there was only one curve the covering map would produce an immersed monogon for $\lambda$ on $\Sigma$, and if the pair of curves at this intersection point were not linked the covering map would produce an immersed bigon for $\lambda$ on $\Sigma$. The self-intersection number of $\Gamma$ is equal to the number of $\pi_1\Sigma$-orbits of linked elevations of curves from $\Gamma$ in the universal cover $\widetilde{\Sigma}$, so we are done.
\end{proof}

In order to recognize the presence of hyperplane separated cubes in the dual cube complex of $\Gamma$, we will employ:

\begin{lemma}
\label{suff condition for separated}
{Suppose that $\Gamma$ has a 
realization $\lambda\subset\Sigma$ without immersed bigons or monogons, and has $k$ points of transverse self-intersection of orders $n_1,\ldots,n_k$, listed with multiplicity. Then $\mathcal{C}(\Gamma)$ has cubes of dimensions $n_1,\ldots,n_k$, with multiplicity. Moreover, the $\pi_1\Sigma$-orbit of the union of these cubes is hyperplane separated for the action of $\pi_1\Sigma$ on $\mathcal{C}(\Gamma)$ if and only if the complement $\Sigma \setminus \lambda$ has no triangles.}
\end{lemma}

\begin{proof}
{Consider the pre-image $\widetilde{\lambda}:=\pi^{-1}\lambda\subset\widetilde{\Sigma}$, and choose lifts $p_1,\ldots,p_k$ of the self-intersection points of $\lambda$, where $p_i$ has order $n_i$. By hypothesis, there are $n_i$ linked elevations from $\widetilde{\lambda}$ through $p_i$, so that there is an $n_i$-cube in $\mathcal{C}(\Gamma)$. We denote this $n_i$-cube corresponding to the choice of lift $p_i$ by $C_i$.

If the complement $\Sigma \setminus \lambda$ had a triangle, then this triangle would lift to $\widetilde{\Sigma}$, so that the curves corresponding to $C_i$, for some $i$, would contain two sides of the lifted triangle. As a consequence, there would be a different lift $p'$ of one of the intersection points, so that $p'$ is also abutting this triangle. Let $C'$ indicate the maximal cube corresponding to the lift $p'$. By construction, $C'$ shares a hyperplane with $C_i$, while there are two other hyperplanes, one of $C_i$ and one of $C'$, that intersect. As $C'$ is in the same $\pi_1\Sigma$-orbit as $C_j$, for some $j$, the union of orbits $\bigcup\pi_1\Sigma\cdot C_i$ is not hyperplane separated.

Finally, suppose the union of orbits is not hyperplane separated. Then there is $g\in\pi_1\Sigma$ so that $C_i$ and $g\cdot C_i$ share a hyperplane $\gamma$, and have a pair of other intersecting hyperplanes, say $\gamma_{1}$ and $\gamma_{2}$. In this case, there is a triangle $T\subset\widetilde{\Sigma}$ formed by $\gamma$, $\gamma_{1}$ and $\gamma_{2}$. While this triangle may not embed under the covering map, it contains an innermost triangle, namely a triangle in the complement of $\widetilde{\Sigma}\setminus\widetilde{\lambda}$. This triangle must embed under the covering map, so that $\Sigma\setminus\lambda$ contains a triangle.}
\end{proof}

\section{Closed curves from ribbon graphs}
\label{ribbon section}

We now prove Theorem \ref{ribbon graph estimate}, obtaining explicit constructions of curves to which Theorem \ref{main estimate} applies. Recall that a \emph{ribbon graph} is a graph with a cyclic order given to the oriented edges incident to each vertex, and a ribbon graph is \emph{even} if the valence of each edge is even. In what follows, we introduce notation for even ribbon graphs and analyze the closed curves that they determine.

Let $S(n)$ indicate the union of the $n$ line segments 
\[ 
\left\{ \ t\exp(\pi im/n):t\in[-1,1] \ \right\}~,
\] 
for $m=1,\ldots,n$, and label the endpoints $\exp(\pi im/n)$ and $-\exp(\pi im/n)$ by $a_{m}$ and $a_{m}'$, respectively. Fix the permutation $\mu$ of the endpoints given by $\mu(a_m)=a_m'=\mu^{-1}(a_m)$, for $m=1,\ldots,n$. We refer to $S(n)$ below as a `star', and $\mu$ as the `switch' map of the star.

Let $\underline{n}$ indicate the tuple $(n_1,\ldots,n_k)$, and consider the union $S\hspace{-.05cm}\left(\underline{n}\right):=S(n_1)\sqcup\ldots\sqcup S(n_k)$. Let $\sigma$ be a fixed-point-free, order two permutation (that is, a `pairing') of the set 
\[
\left\{ \ a_{j,i},a_{j,i}' \ | \ 1 \le j \le n_i, 1 \le i \le k \ \right\}~.
\]
Let $\Gamma\hspace{-.05cm}\left(\underline{n},\sigma\right)$ indicate the graph given by $$\Gamma\hspace{-.05cm}\left(\underline{n},\sigma\right) := S\hspace{-.05cm}\left(\underline{n}\right) / \sim,$$
where $a_{j,i} \sim \sigma (a_{j,i} )$. The vertices of $\Gamma\hspace{-.05cm}\left(\underline{n},\sigma\right)$ are in bijection with the stars $S(n_j)$, and the orientation of $\C$ at $0$ induces a cyclic order to the vertex contained in each $S(n_j)$. These orientations give $\Gamma\hspace{-.05cm}\left(\underline{n},\sigma\right)$ the structure of an even ribbon graph. Moreover, it is clear that every even ribbon graph can be constructed in this way.

Let $\Sigma\hspace{-.05cm}\left(\underline{n},\sigma\right)$ be the surface with boundary associated to the ribbon graph $\Gamma\hspace{-.05cm}\left(\underline{n},\sigma\right)$. We identify $\Gamma\hspace{-.05cm}\left(\underline{n},\sigma\right)$ as smoothly\footnote{Note that the smooth structure of $\Gamma\hspace{-.05cm}\left(\underline{n},\sigma\right)$ in a neighborhood of its vertices is induced by viewing $S(n_j)$ as an immersed submanifold of $\C$.} and incompressibly embedded in $\Sigma\hspace{-.05cm}\left(\underline{n},\sigma\right)$, so that the embedding induces isomorphisms on the level of fundamental groups. By a \emph{closed curve} in $\Gamma\hspace{-.05cm}\left(\underline{n},\sigma\right)$, we mean the free homotopy class of the image of a smooth immersion of $S^{1}$ into $\Gamma\hspace{-.05cm}\left(\underline{n},\sigma\right)$.

\begin{lemma}
\label{closed curve}
{Closed curves in $\Gamma\hspace{-.05cm}\left(\underline{n},\sigma\right)$ are in correspondence with fixed cycles of $(\mu\sigma)^l$, for $l>0$. The closed curves in $\Gamma\hspace{-.05cm}\left(\underline{n},\sigma\right)$ are in minimal position in $\Sigma\hspace{-.05cm}\left(\underline{n},\sigma\right)$, and the total intersection number of these closed curves is given by $\binom{n_1}{2}+\binom{n_2}{2}+\ldots+\binom{n_k}{2}$.}
\end{lemma}

\begin{proof}
The first statement is evident. The second follows from Lemma \ref{min pos immersed bigons}, since the complement $\Sigma\hspace{-.05cm}\left(\underline{n},\sigma\right)\setminus \Gamma\hspace{-.05cm}\left(\underline{n},\sigma\right)$ contains no disks, and hence no immersed bigons or monogons. 
\end{proof}

To obtain closed curves on closed surfaces, one may glue together $\Sigma\hspace{-.05cm}\left(\underline{n},\sigma\right)$ and another (possibly disconnected) surface along its boundary. Some of the components that are glued may be disks, so it is possible that the curves from $\Gamma\hspace{-.05cm}\left(\underline{n},\sigma\right)$ are no longer in minimal position. While \S\ref{immersed bigon section} can be used to build an algorithm that can be applied on a case-by-case basis, a more straightforward control on this phenomenon can be obtained in many cases:

\begin{lemma}
\label{ribbon graphs minimal position}
Suppose that $\hat{\Sigma}$ is a surface obtained by a gluing of $\Sigma\hspace{-.05cm}\left(\underline{n},\sigma\right)$, so that there is a natural inclusion $\Sigma\hspace{-.05cm}\left(\underline{n},\sigma\right) \into \hat{\Sigma}$. If the complement $\hat{\Sigma} \setminus \Gamma\hspace{-.05cm}\left(\underline{n},\sigma\right)$ contains no monogons, bigons, or triangles, then the closed curve $\Gamma\hspace{-.05cm}\left(\underline{n},\sigma\right) \subset \hat{\Sigma}$ is in minimal position.
\end{lemma}

\begin{proof}
Suppose that $\Gamma\hspace{-.05cm}\left(\underline{n},\sigma\right)\subset \hat{\Sigma}$ is not in minimal position, so that by Lemma \ref{min pos immersed bigons} it has either an immersed monogon or bigon. Suppose that $\phi: B\into \hat\Sigma$ is an example of the latter. By assumption, the bigon is not embedded. Thus $\phi^{-1}(\Gamma\hspace{-.05cm}\left(\underline{n},\sigma\right))$ consists of the two sides of $B$, together with some connected arcs that connect opposite sides of the bigon $B$. It is easy to see by induction on the number of such arcs that the complement in $B$ must contain either a bigon or a triangle. This triangle embeds under $\phi$, violating the assumption that $\hat{\Sigma} \setminus \Gamma\hspace{-.05cm}\left(\underline{n},\sigma\right)$ contains no triangles. The case of an immersed monogon is straightforwardly similar.
\end{proof}

Lemmas \ref{closed curve}, \ref{ribbon graphs minimal position} and \ref{suff condition for separated} imply Theorem \ref{ribbon graph estimate} directly. 

\section{A family of examples and Theorem \ref{corollary of main thm}}
\label{examples section}

Towards Theorem \ref{corollary of main thm}, for each $k$ let $\underline{\tau}_k$ indicate the sequence $(3,\ldots,3)$ with $k$ terms. The vertices of $S(\underline{\tau}_k)$ are given by 
\[
\left\{ \ a_{1,j},a_{2,j},a_{3,j},a_{1,j}',a_{2,j}',a_{3,j}' \ | \ 1\le j \le k \ \right\}~.
\]
Let $\sigma$ indicate the following pairing:
\begin{align*}
a_{2,j} & \longleftrightarrow a_{1,j}'~, \\
a_{3,j} & \longleftrightarrow a_{2,j}'~, && \text{for }j=1,\ldots,k, \\
a_{1,j} & \longleftrightarrow a_{3,j+1}'~, && \text{for }j=1,\ldots,k-1, \text{ and} \\
a_{1,k} & \longleftrightarrow a_{3,1}'~.
\end{align*}

See Figure \ref{Gamma_tau_k} for a schematic picture of $\Gamma\hspace{-.05cm}\left(\underline{\tau}_k,\sigma\right)$, and Figure \ref{Gamma_tau_6_gluing} for a gluing of $\Sigma\hspace{-.05cm}\left(\underline{\tau}_6,\sigma\right)$ to which Proposition \ref{suff condition for separated} and Lemma \ref{ribbon graphs minimal position} apply.

\begin{figure}[h]
	\centering
	\includegraphics[width=8cm]{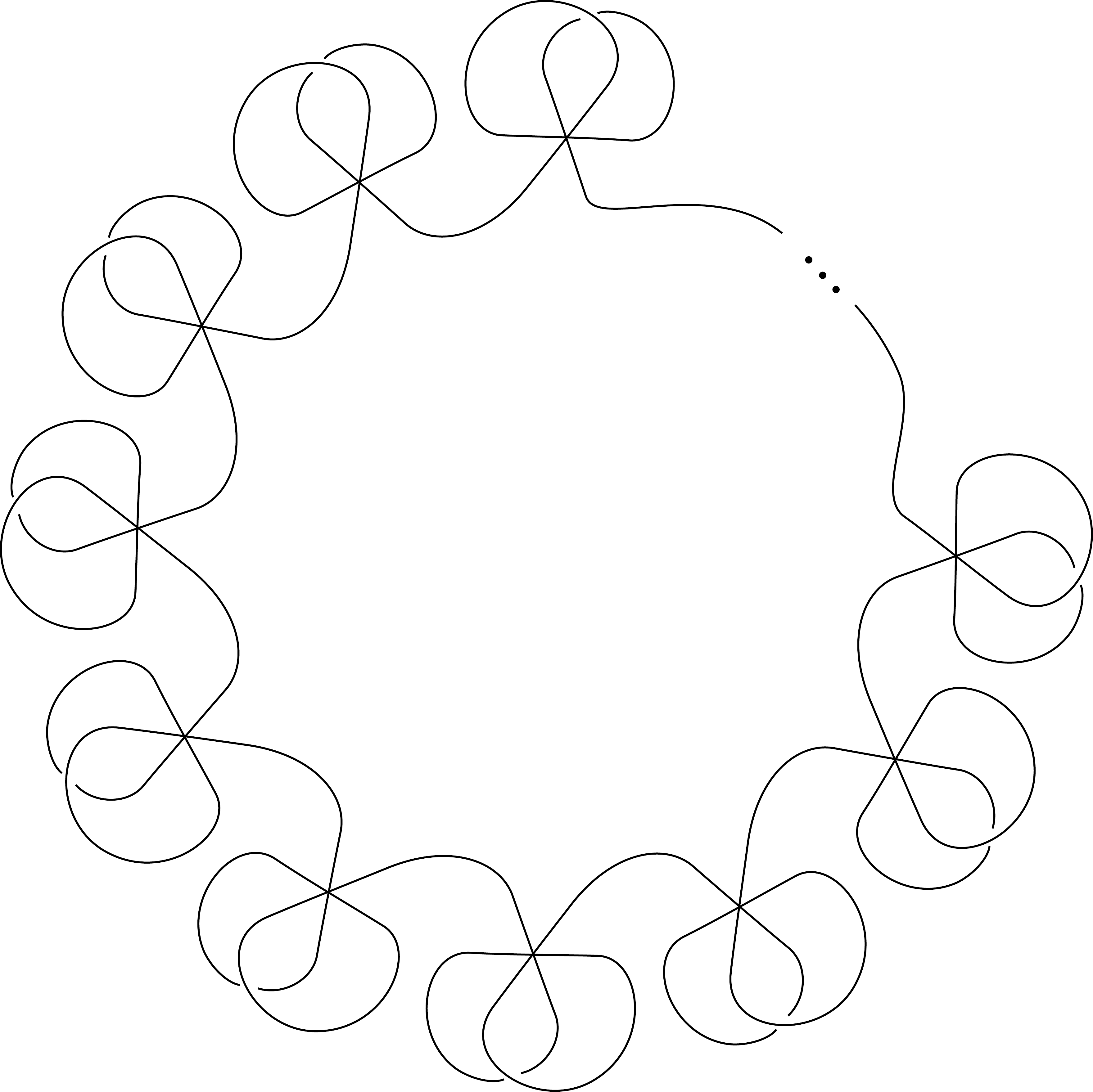}
	\caption{The ribbon graph $\Gamma\hspace{-.05cm}\left(\underline{\tau}_k,\sigma\right)$.}
	\label{Gamma_tau_k}
\end{figure}

\begin{figure}[h]
	\centering
	\includegraphics[width=8cm]{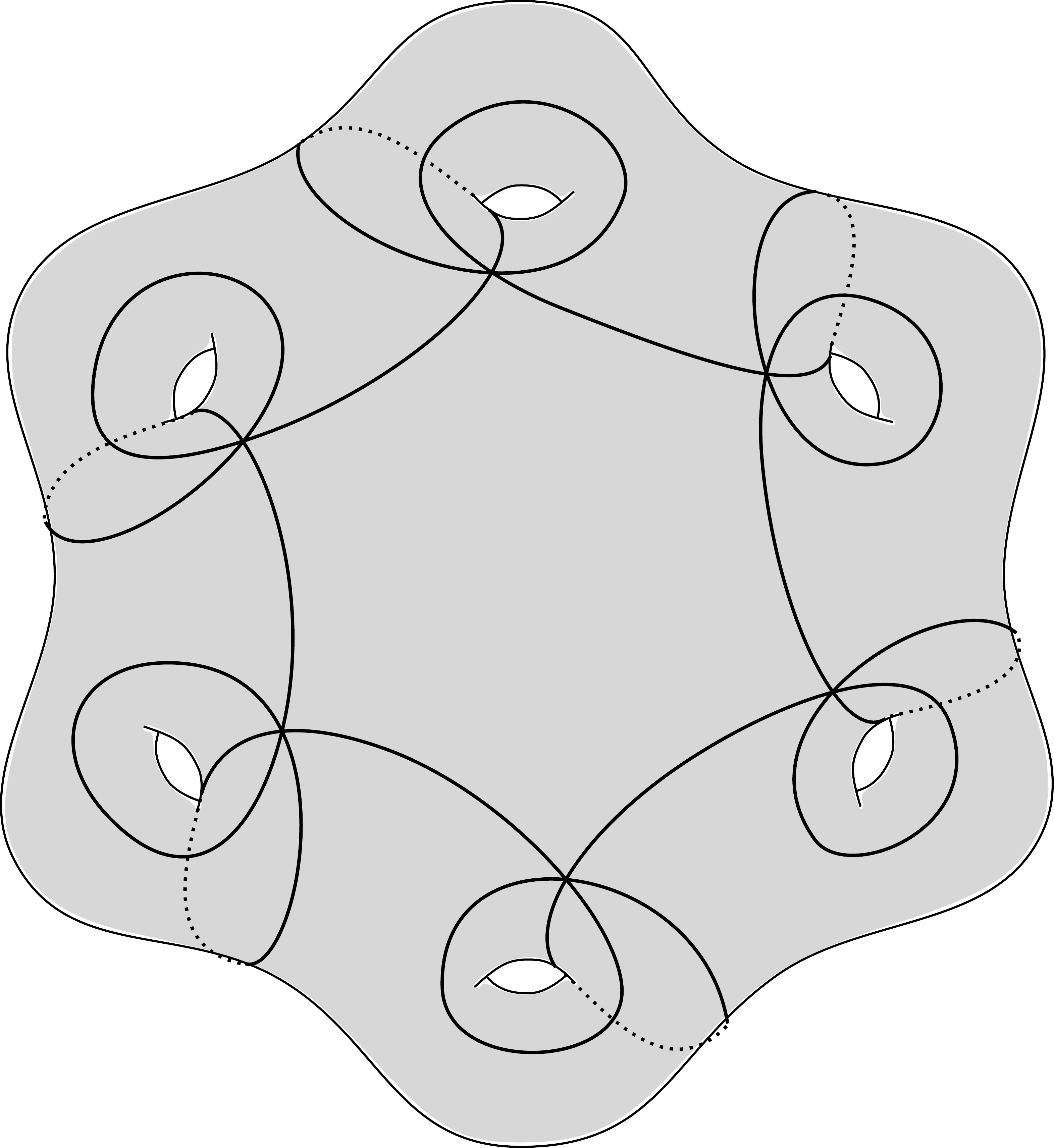}
	\caption{A gluing of $\Sigma\hspace{-.05cm}\left(\underline{\tau}_6,\sigma\right)$ without triangles or bigons in the complement of $\Gamma\hspace{-.05cm}\left(\underline{\tau}_6,\sigma\right)$, so that the given closed curve has six hyperplane separated $3$-cubes.}
	\label{Gamma_tau_6_gluing}
\end{figure}

\begin{proof}[Proof of Theorem \ref{corollary of main thm}]
{Let $\Sigma$ indicate the surface $\Sigma\hspace{-.05cm}\left(\underline{\tau}_k,\sigma\right)$, so that $\Sigma$ contains an embedded minimal position copy of the curve $\Gamma\hspace{-.05cm}\left(\underline{\tau}_k,\sigma\right)$, with the self-intersection number $3k$ by Lemma \ref{closed curve}. 

By Lemma \ref{suff condition for separated}, the dual cube complex of $\Gamma\hspace{-.05cm}\left(\underline{\tau}_k,\sigma\right)$ in $\Sigma$ contains $k$ hyperplane separated 3-cubes. Using Theorem \ref{main estimate}, we may estimate:

\begin{align*}
\limsup_{k\to\infty} \; \frac {\overline{M}_{k}} {k}  &\ge  \limsup_{k\to\infty} \; \frac {\overline{M}_{3k}} {3k} \\
&\ge \limsup_{k\to\infty} \; \frac{1}{3k} \ \inf \left \{ \ \ell(\Gamma\hspace{-.05cm}\left(\underline{\tau}_k,\sigma\right),X) \ : \ X\in\T(\Sigma) \ \right\} \\
&\ge \limsup_{k\to\infty} \; \frac{ k \log \left(\frac{1+\cos\frac{\pi}{3}}{1-\cos\frac{\pi}{3}}\right) } { 3k} = \frac{1}{3}\log 3~.
\end{align*}
}
\end{proof}

Note that in the construction above, it is evident that the genus of $\Sigma\hspace{-.05cm}\left(\underline{\tau}_k,\sigma\right)$ will grow with the self-intersection number of $\Gamma\hspace{-.05cm}\left(\underline{\tau}_k,\sigma\right)$. As a consequence, these lower bounds are not applicable to $\overline{m}_{k}(\Sigma)$ for a fixed surface $\Sigma$. It seems likely\footnote{While this paper was under review, this has been shown in \cite[Theorem 1.4]{agps}.} that $\overline{m}_k(\Sigma)$ grows as $\sqrt{k}$.

\end{document}